\documentclass[12pt]{article}
\usepackage[left=1in,top=1in,right=1in,bottom=1in]{geometry}

\usepackage[english]{babel}
\usepackage{amsmath,amsthm,amsfonts,amssymb,epsfig,xcolor,authblk, hyperref}
\usepackage{bbm,booktabs,float,mathtools,siunitx,tikz,ulem}

\newtheorem{thm}{Theorem}[section]
\newtheorem{lem}[thm]{Lemma}
\newtheorem{prop}[thm]{Proposition}
\newtheorem{cor}[thm]{Corollary}

\newcommand{\E}{\mathbf{E}}

\newcommand{\PP}{\mathbf{P}}

\newcommand{\R}{\mathbb{R}}
\newcommand{\Z}{\mathbb{Z}}
\newcommand{\N}{\mathbb{N}}
\newcommand{\e}{\varepsilon}
\newcommand{\ii}{{\rm i}}
\newcommand{\oo}{{\rm o}}
\newcommand{\rd}{{\rm d}}
\newcommand{\m}{\text{\rm meas}}

\DeclareMathOperator{\OO}{O}

\newcommand{\re}{{\rm Re}}

\newcommand{\1}{\mathbf 1}

\numberwithin{equation}{section}

\title{\Large Conditional Upper Bounds for Large Deviations and Moments of the Riemann Zeta Function}

\author[1,2]{\normalsize Louis-Pierre Arguin}
\author[3]{\normalsize Emma Bailey}
\author[3]{\normalsize Asher Roberts}

\affil[1]{\footnotesize  \it Department of Mathematics, Baruch College and Graduate Center, City University of New York, NY}
\affil[2]{\footnotesize  \it Mathematical Institute, University of Oxford, Oxford, UK}
\affil[3]{\footnotesize  \it Department of Mathematics, University of Bristol, Bristol, UK}
\affil[4]{\footnotesize  \it Department of Mathematics, St. Joseph's University, New York, NY}

\date{April 28, 2026}

\begin{document}

\maketitle
\begin{abstract}
Assuming the Riemann Hypothesis, we show that for $k>0$
$$
\frac{1}{T}\m\Big\{t\in [T,2T]:|\zeta(1/2+\ii t)|>(\log T)^k\Big\}\leq C_k \frac{(\log T)^{-k^2}}{\sqrt{\log\log T}},
$$
where $C_k=\exp(e^{ck})$ for some absolute constant $c>0$.
This implies that the $2k$-moments of $|\zeta|$ are bounded above by $C_k(\log T)^{k^2}$, recovering
the bound of \cite{Har13}.
The proof relies on the recursive scheme of \cite{ArgBouRad20}, and combines ideas of \cite{Sou09} and \cite{Har13}. 
\end{abstract}


\section{Introduction}

\subsection{Main Result}
The Riemann zeta function is defined in the region $\re \ s>1$ of $\mathbb C$ by 
\begin{equation}
\zeta(s)=\sum_{n\geq 1}n^{-s}. 
\end{equation}
It admits a meromorphic continuation to the complex plane with a pole at $s=1$ and satisfies the functional equation $\zeta(s)=2^s\pi^{s-1}\sin(\pi s/2) \Gamma(1-s)\zeta(1-s)$. 
The function thus admits trivial zeros at negative even integers. The Riemann Hypothesis asserts that all other zeros must lie on the critical line $\re \ s=1/2$.

This paper is concerned with the distribution of large values of the zeta function occurring on the critical line. 
The main result is an upper bound on the measure of level sets of $\zeta$ in the interval $[T,2T]$ for values of the order $\log\log T$, conditionally on the Riemann Hypothesis:
\begin{thm}
\label{thm: main}
Assume the Riemann Hypothesis is true. Let $k>0$ and $V=V(T)\sim k \log\log T$. We have for $T$ large enough
\[
\frac{1}{T}\m\Big\{t\in [T,2T]:\log|\zeta(1/2+\ii t)|>V\Big\}\leq C_k \frac{e^{-V^2/\log\log T}}{\sqrt{\log\log T}},
\]
where $C_k= \exp(e^{ck})$ for some absolute constant $c>0$.
\end{thm}
An upper bound of this form was proved unconditionally on the range $0<k<2$ in \cite{ArgBai23}.
A matching lower bound with constant $(k^2\log k)^{-k^2+\oo(1)}$ is proved in \cite{ArgCre26} for all $k>0$. 
These results are to be contrasted with Selberg's Central Limit Theorem which states that for any $v\in \R$
$$
\lim_{T\to\infty}\frac{1}{T}\m\Big\{t\in [T,2T]:\log|\zeta(1/2+\ii t)|>v\sqrt{\tfrac{1}{2}\log\log T}\Big\}=\int_v^{\infty} \frac{e^{-x^2/2}}{\sqrt{2\pi}}\rd x.
$$
Theorem \ref{thm: main} and the corresponding lower bound show that Gaussian fluctuations extend to large deviations for the whole range of order $\log\log T$, up to constant.
The interest for this regime of values stems from its relation to the moments. 
In the seminal paper \cite{Sou09}, Soundararajan proved the following bound for $k\geq 0$ under the Riemann Hypothesis:
\begin{equation}
\label{eqn: sound m}
\frac{1}{T}\m\Big\{t\in [T,2T]: |\zeta(1/2+\ii t)|>(\log T)^k\Big\}=(\log T)^{-k^2+\oo(1)}.
\end{equation}
Since the measure of level sets and the $2k$-moments satisfy the relation
\begin{equation}
\label{eqn: IBP}
\int_T^{2T} |\zeta(1/2+\ii t)|^{2k}\rd t
=2k\int_{-\infty}^\infty e^{2kV} \m\Big\{t\in [T,2T]: |\zeta(1/2+\ii t)|>e^V\Big\}\rd V,
\end{equation}
the bound \eqref{eqn: sound m} then implies that the $2k$-moment is $\leq (\log T)^{k^2+\e}$ for any fixed $\e>0$, with the dominant contribution coming from $V$ around $k\log\log T$.
In the breakthrough paper \cite{Har13}, Harper sharpened this bound to
$$
\frac{1}{T}\int_T^{2T} |\zeta(1/2+\ii t)|^{2k}\rd t\leq e^{e^{\OO(k)}} (\log T)^{k^2}.
$$
(The effective constant can be taken to be $e^{e^{18.63k}}$ as shown in \cite{Tao24}.)
Theorem \ref{thm: main} and the relation \eqref{eqn: IBP} imply Harper's bound:
\begin{cor}
\label{cor: moments}
Assume the Riemann Hypothesis is true. Let $k\geq 2$. We have for $T$ large enough
$$
\frac{1}{T}\int_T^{2T} |\zeta(1/2+\ii t)|^{2k}\rd t\ll e^{e^{ck}}(\log T)^{k^2},
$$
for the same absolute constant $c$ as Theorem \ref{thm: main}.
\end{cor}
The proof is identical to \cite[Corollary 1.2]{ArgBai23} with Theorem \ref{thm: main} as the new input.
Unconditionally, upper bounds for $0\leq k\leq 2$ were proved in \cite{HeaRadSou19}, as well as lower bounds for the whole range $k\geq 0$ in \cite{HeaSou22}.
The {\it Moments Conjecture} of Keating \& Snaith  \cite{KeaSna00} predicts that the moments are asymptotic to $a_kg_k (\log T)^{k^2}$ with explicit arithmetic constant $a_k$ and geometric constant $g_k$. 
Only the cases $k=1$ and $k=2$ are known from the classical works of Hardy \& Littlewood \cite{harlit18} and Ingham~\cite{ing26}. 
As argued in \cite{Rad11}, the optimal $C_k$ in Theorem \ref{thm: main} should also be  $a_kg_k$. 
Note that as $k\to\infty$, we have the asymptotic behavior (see for example \cite{ConGon01, HiaRub11})
\begin{equation}
\label{eqn: KS}
a_k=(\log k)^{-k^2+\oo(1)},\qquad g_k=k^{-k^2+\oo(1)}.
\end{equation}
Hence the constant in Theorem \ref{thm: main} is still far off the predicted value as it should decay to $0$ sharply as $k\to\infty$.
For more information on the moments problem, we refer to the survey of Florea \cite{Flo26}, and to the review paper of Soundararajan  \cite{Sou23} for the broader question of the distribution of values of $L$-functions. In Section \ref{sect: short}, we discuss the consequences of Theorem \ref{thm: main} to the statistics of large values on short intervals of the critical line.\\

As in \cite{Sou09} and \cite{Har13}, the proof of Theorem \ref{thm: main} relies on the following conditional pointwise bound of Soundararajan:
\begin{prop}[\cite{Sou09}]
\label{prop: Sound}
Assume the Riemann Hypothesis is true. Let $T$ be large enough and $t\in [T,2T]$. Consider $2\leq x\leq T^2$. Let $\lambda_0=0.4912\dots$ denote the unique positive real solution to $e^{-\lambda}=\lambda+\lambda^2/2$. For all $\lambda\geq \lambda_0$ we have
\begin{equation}
\label{eqn: sound bound}
\log|\zeta(1/2+\ii t)|\leq \re\ \sum_{n\leq x}\frac{a(n)\Lambda(n)}{n^{\sigma+\ii t}\log n}+\frac{(1+\lambda)}{2}\cdot\frac{\log T}{\log x} +\OO\Big(\frac{1}{\log x}\Big),
\end{equation}
where $a(n)=\frac{\log (x/n)}{\log x}$ and $\sigma=\sigma(x,\lambda)=\frac{1}{2}+\frac{\lambda}{\log x}$. 
\end{prop}
The main challenge in estimating with good accuracy the moments or the large deviations is to take $x$ large in the above bound, while maintaining sufficient control of the Dirichlet polynomial to prove Gaussian behavior.
In \cite{Sou09}, this was done up to $x=T^{\frac{1}{2k^2}\frac{\log_3T}{\log_2T}}$, and $x=T^{e^{-1000k}}$ in \cite{Har13}. 
As explained in the next section, we rely on a random walk heuristic based on the recursive scheme of \cite{ArgBouRad20} which allows for a control of the partial sums up to a threshold similar to  \cite{Har13}.
The asymptotic \eqref{eqn: KS} for $g_k$ would be a consequence of Gaussian fluctuations in the range of primes between $T^{1/k^2}$ and $T^{1/k}$, but this seems currently out of reach.

\subsection{The Recursive Scheme}
Throughout the paper, we will use the probabilistic notation where $\tau$ denotes a random variable uniformly distributed on $[T,2T]$, with $\PP$ standing for the corresponding distribution and $\E$, for the expectation.
We will also follow the probabilistic convention of omitting the dependence on $\tau$ in random variables and events. 
In this notation, the proof of Theorem \ref{thm: main} amounts to estimating
for fixed positive $k$ and $V\sim k\log\log T$, the probability $\PP(H)$ of the event
$$
H=\{\log |\zeta (1/2+\ii \tau)|>V\}.
$$

The recursive scheme of \cite{ArgBouRad20} consists of decomposing $H$ into decreasing good events $G_\ell$, $\ell\geq 1$, so that
\begin{equation}
\label{eqn: recursive scheme}
\PP(H)=\sum_{\ell=1}^{\mathcal L+1} \PP(H\cap G_{\ell-1}\setminus G_{\ell}),
\end{equation}
with the convention that $G_0=[T,2T]$ and $G_{\mathcal L+1}=\emptyset$. 
The events restrict the values of the partial sums at checkpoints indexed by $1\leq \ell\leq \mathcal L$ defined by
$T_0=1$, $T_1=\exp(\sqrt{\log T})$ and
\begin{equation}
\label{eqn: Tell}
T_\ell=T^{\beta_\ell},\quad \beta_\ell=\frac{e^{\ell-1}}{\sqrt{\log T}}, \quad 1\leq \ell\leq \mathcal L, 
\end{equation}
with
$$
\mathcal L=1+\max\{\ell:\beta_\ell\leq e^{-10^4 k}\}.
$$
With Proposition \ref{prop: Sound} in mind, we define for $\lambda=1/2$, $1\leq j,\ell\leq \mathcal L$, the partial sums
\begin{equation}
\label{eqn: S}
S_\ell^{(j)}=\re \sum_{p\leq T_\ell}\frac{a_{j}(p)}{p^{\sigma_j+\ii \tau}}+\re \frac{1}{2}\sum_{p^2\leq T_\ell}\frac{a_{j}(p^2)}{p^{2\sigma_j+2\ii \tau}},
\end{equation}
where
$$
\sigma_j=\frac{1}{2}+\frac{\lambda}{\log T_j}\qquad a_j(n)=\frac{\log (T_j/n)}{\log T_j}.
$$
The moments of $S_\ell^{(j)}$ are estimated in Appendix \ref{sect: moments}. In particular, the variance is approximately $\frac{1}{2}\log\log T_\ell$.
It is expected that the random phases $p^{-i\tau}$ become essentially independent in the limit of large $T$ for distinct primes, thus that the partial sums ranging in $\ell$ behave like random walks with time index $t_\ell=\log \log T_\ell$.
Note that each abscissa $\sigma_j$ gives rise to a random walk $(S_\ell^{(j)}, 1\leq \ell\leq \mathcal L)$, which complicates the analysis.
Using this heuristic as a guide, if $\log |\zeta|$ achieves a value of $V\sim k\log\log T$, then the value of $S_\ell^{(j)}$ at $T_\ell$ is most likely very close to $\kappa\log\log T_\ell$
where $\kappa$ is the gradient
$$
\kappa=\frac{V}{\log\log T}.
$$
The good events $G_\ell$ for $1\leq \ell\leq \mathcal L$ are defined as to reflect the random walk behavior:
\begin{equation}
\label{eqn: G}
G_\ell=\{S_\ell^{(j)}\in [L_\ell, U_\ell],\ \forall j\geq \ell\}\cap G_{\ell-1},
\end{equation}
with
\begin{equation}
\label{eqn: barriers}
L_\ell=\kappa\log\log T_\ell -c_\ell, \qquad U_\ell=\kappa \log\log T_\ell+c_\ell.
\end{equation}
The parameters $c_\ell$ capture the fluctuations around the linear interpolation.
For a random walk achieving a value $V$ at time $\frac{1}{2}\log\log T$, the fluctuations at time $\frac{1}{2}\log\log T_\ell$ are of the order 
$$
c_\ell\approx \Big( \log\log T-\log\log  T_\ell\Big)^{1/2+\varepsilon}\approx (\log \beta_{\ell}^{-1})^{1/2+\e}.
$$
As seen from the proof, our current approach forces us to take a larger quantity
\begin{equation}
\label{eqn: Bell}
 c_\ell=\beta_\ell^{-\gamma},\quad \gamma=1/25.
\end{equation}
(The analysis allows for any $\gamma$ smaller than $1/20$, cf.~Equation \eqref{eqn: key condition}.) 
Note that $c_{\ell-1}=e^{\gamma}c_\ell>c_\ell$, narrowing the barriers at each subsequent checkpoint.
With these definitions, we have the following main proposition:
\begin{prop}
\label{prop: good}
Assume the Riemann Hypothesis is true. Let $k>0$ and $V=V(T)\sim k \log\log T$.
For all $1\leq \ell\leq \mathcal L$, we have for some absolute constant $c>0$ and $T$ large enough
\begin{align}
\PP(H\cap  G_{\ell-1}\setminus G_\ell)&\leq e^{-c_\ell} \cdot \frac{e^{-V^2/\log\log T}}{\sqrt{\log\log T}} \label{eqn: small},\\
\PP(H\cap  G_{\mathcal L})&\leq  e^{e^{c k}}\cdot \frac{e^{-V^2/\log\log T}}{\sqrt{\log\log T}} \label{eqn: large}.
\end{align}
\end{prop}
\begin{proof}[Proof of Theorem \ref{thm: main}]
The proof is straightforward from the above proposition, after noticing that $e^{-c_\ell}$ is summable in $\ell$. The dominant term in the sum is \eqref{eqn: large}.
\end{proof}
The proof of Proposition \ref{prop: good} uses an approximation of indicator functions in terms of Dirichlet polynomials, as detailed in Appendix \ref{sect: GC}. 
The left-hand side of \eqref{eqn: small} is smaller than
\begin{equation}
\label{eqn: split}
\PP\Big( G_{\ell-1}\cap \{\exists j\geq \ell: S_{\ell}^{(j)}>U_{\ell}\}\Big)+\PP\Big(H\cap G_{\ell-1}\cap \{\exists j\geq \ell: S_{\ell}^{(j)}<L_{\ell}\}\Big).
\end{equation}
We refer to these events as the `upper' and `lower' barrier bounds respectively. The main new input of this paper is the estimate of the second term, which is done in Section \ref{sect: LB}. 
In \cite{ArgBai23}, a similar estimate was handled using a twisted fourth moment, which is insufficient when $k\geq 2$. 
We rely instead in Section \ref{sect: LB} on the bound \eqref{eqn: sound bound}.
An {\it a priori} bound is needed for this estimate and is proved in Section \ref{sect: AB}. The first term in \eqref{eqn: split} is dealt with similarly in Section \ref{sect: UB}.
In each case, we estimate the probability for fixed $j\geq \ell$ and use a union bound on $j$. The total contribution will be negligible since 
\begin{equation}
\label{eqn: j}
\#\{j\geq \ell\}=\mathcal L-\ell+1=\log \frac{\beta_\mathcal L}{\beta_\ell}+1\leq \log \beta_\ell^{-1}+1,
\end{equation}
which will be absorbed by the probability yielding  \eqref{eqn: small}.

As in \cite{Sou09, Har13}, we expect that the method of proof of Theorem \ref{thm: main} can be adapted to yield sharp bounds for the large deviations and moments of certain families of Dirichlet $L$-functions.
\\

\noindent{\bf Notation}.
We will use Vinogradov's notation with a slight convenient twist. Let $f,g$ functions of $T$. For $k$ fixed, we write $f\ll g$ whenever 
\begin{equation}
\label{eqn: convention}
\limsup_{T\to\infty}|f/g|\leq e^{c k^2},\ \text{ for some absolute constant $c>0$.}
\end{equation}
We write $f\asymp g$ if $f\ll g$ and $f\gg g$. 
To simplify the notation, we will sometimes in the proof express the checkpoints in the $\log\log$-scale, i.e., 
$$
t_\ell=\log\log T_\ell.
$$

\noindent{\bf Acknowledgements}.
The authors thank Nathan Creighton for his helpful comments on the first draft of this paper. 
L.-P. Arguin is supported by the grants NSF DMS 2153803 and EPSRC EP/Z535990/1. 
 For the purpose of open access, the authors have applied a CC
BY public copyright licence to any author accepted manuscript arising from this
submission.

\section{Proof of Proposition \ref{prop: good}}
We assume throughout that $k\geq 2$. We consider $V\sim k\log\log T$ and we write  $\kappa=V/\log\log T$ for the exact gradient. 
The estimate \eqref{eqn: small} is done by splitting the probability into two estimates for the upper and lower barrier as in  \eqref{eqn: split}.
The results are respectively given in Lemma \ref{lem: UB} and \ref{lem: LB} below. For an accurate estimate, an {\it a priori} bound in Lemma \ref{lem: a priori} is first needed.

\subsection{A Priori Bound}
We start by proving a bound for the complement of the event
\label{sect: AB}
$$
A_\ell= \{|S_\ell^{(j)}-S_{\ell-1}^{(j)}|\leq  2c_\ell,\ \forall j\geq \ell \}, \quad 1\leq \ell\leq \mathcal L.
$$
\begin{lem}
\label{lem: a priori}
For $1\leq \ell \leq \mathcal L$, we have
$$
\PP(G_{\ell-1}\cap A_\ell^{\rm c})\ll e^{-c_\ell}\cdot \frac{e^{-V^2/\log\log T}}{\sqrt{\log\log T}}.
$$
\end{lem}

\begin{proof}
As noticed in \eqref{eqn: j}, 
it suffices to bound the probability uniformly in $j$. Henceforth, we omit the notation by writing $S_\ell^{(j)}=S_\ell$.
We first evaluate $\PP(A_1^{\rm c})$, which is simply
$$
\PP(A_1^{\rm c})\ll \log\beta_1^{-1}\cdot \PP(|S_1|> 2c_1)\leq \frac{\log\beta_1^{-1}}{(2c_1)^{2q}}\cdot  \E[|S_1|^{2q}],
$$
using Markov's inequality for some $q>1$. Equation \eqref{eqn: moment bound sqrt} of the appendix is applicable with the choice of $q=\lceil 4c_1^2\rceil$ since $2q \beta_1\leq 10 \beta_1^{1-2\gamma}\leq 1/100$ for $T$ large enough.
This gives
$$
\PP(A_1^{\rm c})\ll\sqrt{\log\log T}\cdot e^{-4c_1^2}\ll e^{-c_1}\cdot  \frac{e^{-V^2/\log\log T}}{\sqrt{\log\log T}}.
$$

For $\ell>1$ we decompose the event in terms of the value of $S_{\ell-1}$, so that
$$
\PP(G_{\ell-1}\cap A_\ell^{\rm c})\ll  (\log \beta_{\ell}^{-1}) \sum_{u\in [L_{\ell-1},U_{\ell-1}]}\PP\Big(\{S_\ell-S_{\ell-1}> 2c_\ell, S_{\ell-1}\in [u-1,u]\}\cap G_{\ell-1}\Big).
$$
By writing
\begin{equation}
\label{eqn: Gu}
G_{\ell-1}(u)=G_{\ell-2}\cap\{S_{\ell-1}\in [u-1,u]\},
\end{equation}
the above is
$$
\ll  (\log \beta_{\ell}^{-1})\sum_{u\in [L_{\ell-1},U_{\ell-1}]} \E\left[\frac{|S_{\ell}-S_{\ell-1}|^{2q}}{(2c_\ell)^{2q}} \mathbf{1}(G_{\ell-1}(u))\right].
$$
We choose $q=\lceil 4c_\ell^2\rceil$. The polynomial $(S_{\ell}-S_{\ell-1})^q$ has then length at most $T^{2q\beta_\ell}$ which is $\leq T^{1/100}$ for $T$ large enough.
We now apply Lemma \ref{lem: 1 point} of the appendix to get
$$
\begin{aligned}
\PP(G_{\ell-1}\cap A_\ell^{\rm c})&\ll (\log \beta_{\ell}^{-1}) \cdot \sum_{u\in [L_{\ell-1},U_{\ell-1}]} \E\left[\frac{|S_{\ell}-S_{\ell-1}|^{2q}}{(2c_\ell)^{2q}}\right]\cdot \frac{e^{-u^2/t_{\ell-1}}}{\sqrt{\log\log T}}\\
&\ll  e^{-4c_\ell^2}\cdot (\log \beta_{\ell}^{-1})\cdot c_{\ell-1}\cdot \frac{e^{-L_{\ell-1}^2/t_{\ell-1}}}{\sqrt{\log\log T}},
\end{aligned}
$$
where we used a trivial bound over $u$ and Equation \eqref{eqn: moment bound}.
Since $L_{\ell-1}=\kappa t_{\ell-1}-c_{\ell-1}$, this is
$$
\begin{aligned}
&\ll  \exp\Big(\kappa^2(t-t_{\ell-1})+2\kappa c_{\ell-1}-c_{\ell-1}^2\Big)\frac{e^{-V^2/\log\log T}}{\sqrt{\log\log T}}\ll e^{-c_\ell}\cdot  \frac{e^{-V^2/\log\log T}}{\sqrt{\log\log T}} ,
\end{aligned}
$$
as claimed.
\end{proof}

\subsection{Upper Barrier}
\label{sect: UB}
We observe that on the event $G_{\ell-1}\cap A_\ell$, $1\leq \ell\leq \mathcal L$, we have
\begin{equation}
\label{eqn: prime}
S_\ell^{(j)} \in [L_\ell', U_\ell '], \ \forall j\geq \ell,
\end{equation}
for the slightly more forgiving barriers
\begin{equation}
\label{eqn: barriers prime}
L'_\ell=\kappa t_{\ell} - 4c_\ell\qquad U'_\ell=\kappa t_{\ell} +4c_\ell.
\end{equation}
 Lemma \ref{lem: UB} and Lemma \ref{lem: LB} restore the uniformity of the restrictions in $\ell$.
We prove:
\begin{lem}
\label{lem: UB}
For $1\leq \ell \leq \mathcal L$, we have
$$
\PP\big(G_{\ell-1}\cap A_\ell\cap \{\exists j\geq \ell: S_\ell^{(j)} >U_\ell\}\big)\ll e^{-c_\ell}\cdot  \frac{e^{-V^2/\log\log T}}{\sqrt{\log\log T}}.
$$
\end{lem}

\begin{proof}
As in the proof of Lemma \ref{lem: a priori}, we do an union bound on $j$, using \eqref{eqn: j}, and write $S_\ell^{(j)}=S_\ell$. The case $\ell=1$ is identical to the one in Lemma \ref{lem: a priori}.
For $\ell>1$, using Lemma \ref{lem: a priori}, we only have to prove the case when $S_\ell\in [U_\ell, U_\ell']$.
We can apply Lemma \ref{lem: 1 point} to get that the probability in question is
$$
\ll  (\log \beta_{\ell}^{-1})\sum_{u\in [U_\ell,U_{\ell}']} \frac{e^{-u^2/t_\ell}}{\sqrt{t_\ell}}\ll (\log \beta_{\ell}^{-1})\cdot c_{\ell-1}\cdot \frac{e^{-U_\ell^2/t_\ell}}{\sqrt{t_\ell}}\ll e^{-c_\ell}\cdot \frac{e^{-V^2/\log\log T}}{\sqrt{\log\log T}} .
$$
This proves the lemma.
\end{proof}

\subsection{Lower Barrier}
\label{sect: LB}
We prove the bound for the second term in \eqref{eqn: split}.
\begin{lem}
\label{lem: LB} 
For $1\leq \ell \leq \mathcal L$, we have
\[
\PP\big(H \cap G_{\ell-1} \cap A_\ell\cap \{\exists j\geq \ell: S_\ell^{(j)} <L_\ell\} \big)\ll e^{-c_\ell}\cdot  \frac{e^{-V^2/\log\log T}}{\sqrt{\log\log T}}.
\]
\end{lem}
\begin{proof}
Proceeding as in the proof of Lemma \ref{lem: a priori}, a union bound on $j$ allows us to consider some fixed $j\geq \ell$. It will be beneficial to retain the superscript notation on $S_\ell^{(j)}$ in the following, see in particular the argument around \eqref{eq:lowerbound_truncation}.

We condition on the position at time $\ell$, and find
  \begin{align*}
    \PP(H\cap G_{\ell-1}\cap A_\ell\cap\{S_\ell^{(j)}<L_\ell\})
    \ll \sum_{u\in [L_\ell',L_\ell]}\PP(\{\log|\zeta(1/2+\ii\tau)|-S_\ell^{(j)}>V-u\}\cap G_{\ell}(u)),
  \end{align*}
using \eqref{eqn: prime} and the notation \eqref{eqn: Gu}. For $\ell=1$, we take $G_1(u)=\{S_1^{(j)}\in[u-1,u]\}$. 

  Now, we use Proposition~\ref{prop: Sound} to replace $\log|\zeta(1/2+\ii\tau)|$ with a Dirichlet polynomial of length $x$ to be determined. 
  This truncation ought to be dependent on the magnitude of the deviation $V-u$. 
  We take $x=T^{\beta_m}$ where $m$ is the smallest index such that
 \begin{equation}\label{eq:truncation_length}
  \beta_m>c_\ell^{-1}=\beta_\ell^{\gamma}.
 \end{equation}
  Applying Proposition~\ref{prop: Sound}, we hence find
  \begin{equation}
  \begin{aligned}
    \{\log|\zeta(1/2+\ii \tau )|-&S_\ell^{(j)}>V-u\}\nonumber\\
    &\subset \Big\{\re\sum_{n\leq T^{\beta_m}}\frac{a(n)}{n^{\sigma_m+\ii \tau}}\frac{\Lambda(n)}{\log n} -S_\ell^{(j)} \geq\frac{1}{4} \beta_m^{-1} +\OO(1/\sqrt{\log T})\Big\}.
    \label{eq:lowerbound_truncation}
  \end{aligned}
  \end{equation}
   For the right-hand side event, the powers of primes greater than $2$ can be absorbed into the lower bound. It follow that
  \[\Big\{\log|\zeta(1/2+\ii \tau )|-S_\ell^{(j)}>V-u\Big\}\subset \Big\{S_m^{(m)} -S_\ell^{(j)} \geq \frac{1}{5}(V-u)\Big\}.\]

  To resolve the discrepancy between the abscissae $\sigma_j$ and $\sigma_m$ when comparing $S_{m}^{(m)}$ and $S_\ell^{(j)}$, we first trivially write
  \[S_m^{(m)} -S_\ell^{(j)} = S_m^{(m)}-S_\ell^{(m)}+(S_\ell^{(m)}-S_\ell^{(j)}).\]
This implies
\begin{equation}
\label{eqn: split W}
  \begin{aligned}
   & \PP\Big(\{S_m^{(m)} -S_\ell^{(j)} \geq \frac{1}{5}(V-u)\}\cap { G_{\ell}(u)}\Big)\leq \\
    &\PP\Big(\{S_m^{(m)}-S_\ell^{(m)}>\frac{1}{10}(V-u)\}\cap{ G_{\ell}(u)}\Big)+ \PP\Big(\{S_\ell^{(m)}-S_\ell^{(j)}>\frac{1}{10}(V-u)\} \cap{G_{\ell}(u)}\Big).
  \end{aligned}
  \end{equation}

  We begin with bounding the first probability on the right-hand side.  In this case, the abscissae are identical. The proof then follows similarly to Lemma~\ref{lem: UB}, using additionally that the Dirichlet polynomials $S_m^{(m)}-S_\ell^{(m)}$ and $S_\ell^{(j)}$ are supported on disjoint sets of primes. Indeed, we have using a Markov inequality and Lemma~\ref{lem: 1 point} that for $W=(V-u)/10$
  \begin{equation}
  \label{eqn: prob LB}
    \PP(\{S_m^{(m)}-S_\ell^{(m)}>W\}\cap G_{\ell}(u))\ll \mathbb{E}\left[\frac{|\mathcal{Q}|^2}{W^{2q}}\right]\frac{e^{-u^2/t_\ell}}{\sqrt{t_\ell}},
  \end{equation}
  provided $\mathcal{Q}=(S_{m}^{(m)}-S_\ell^{(m)})^q$ has length not exceeding $T^{1/2}$. The length is then at most $T^{2q\beta_m}$ so the choice $q=\lceil W/5\rceil$ is permissible since for $u\in[L_\ell',L_\ell]$ we have, using \eqref{eqn: barriers prime} and \eqref{eq:truncation_length}
  \[2q\beta_m <\frac{2e}{50}\frac{(V-u)}{c_\ell}<\frac{2e}{50}\left(\kappa\beta_\ell^{\gamma}\log\beta_\ell^{-1}+4\right)<\frac{1}{2}\ .\]
 An application of Lemma~\ref{lem: Gaussian moments real} and Stirling's approximation yields
  \[\mathbb{E}\left[\frac{|\mathcal{Q}|^2}{W^{2q}}\right] \ll \exp\left(q\log q -q-q\log W^2+q\log(t_m-t_\ell)\right)\ll \exp\left(-\frac{1}{10}W\log W\right),\]
where we used the fact that $W^{1/2}>t_m-t_\ell$. 
Writing $\bar u=u-\kappa t_\ell$, the first probability in \eqref{eqn: split W} is then
  \begin{align*}
    \ll&\frac{e^{-\kappa^2/t_\ell}}{\sqrt{t_\ell}}\sum_{\bar u\in[-4c_\ell,-c_\ell]}  e^{2\kappa |\bar u|}\cdot e^{-\frac{1}{100}|\bar u|\log |\bar u|}
    \ll\frac{e^{-V^2/\log\log T}}{\sqrt{\log\log T}}\cdot e^{\kappa^2(t-t_\ell)}\sum_{\bar{u}\in[-4c_\ell, -c_\ell]} e^{2k |\bar u|-\frac{1}{100}|\bar u|\log |\bar u|}.
  \end{align*}
  The exponent $2k - \frac{1}{100}\log|\bar{u}|$ is negative when $\bar{u}>e^{200k}$.  Recall that $\bar{u}\in[c_\ell,4c_\ell]$, so indeed $c_\ell\geq c_{\mathcal{L}} = \beta_{\mathcal{L}}^{-\gamma} >e^{200k}$ by the choice of $\mathcal{L}$. Therefore  the contribution from the first event in the right-hand side of \eqref{eqn: split W} is
  \[\ll e^{-c_\ell}\cdot  \frac{e^{-V^2/\log\log T}}{\sqrt{\log\log T}}.\]
  
  Returning to the second probability in \eqref{eqn: split W}, we introduce a further decomposition and bound as
  \begin{align*}
    \ll\sum_{\substack{u\in[L_\ell',L_\ell], w\in[L_\ell', U'_\ell]\\|w-u|>\frac{1}{10}(V-u)}}\PP(\{S_\ell^{(m)}\in[w-1,w]\}\cap\{S_\ell^{(j)}\in[u-1,u]\}\cap G_{\ell-1}).
  \end{align*}
  This is now in the form to apply the two-point estimate of Lemma~\ref{lem: 2 point}. This yields
  $$
  \begin{aligned}
  \ll 
  \frac{1}{\sqrt{t_\ell}}\sum_{\substack{u\in[L_\ell',L_\ell], w\in[L_\ell', U'_\ell]\\|w-u|>\frac{1}{10}(V-u)}}e^{-\frac{(u+w)^2}{4t_\ell}}e^{-\frac{1}{8}(w-u)^2}
 &   \ll (\log\beta_\ell^{-1})\cdot \frac{e^{-{L_\ell'}^2/t_\ell}}{\sqrt{t_\ell}}\cdot e^{-c_\ell^2/800},
  \end{aligned}
  $$
  using that $V-u>c_\ell$ for $u\in [L_\ell',L_\ell]$. This is $\ll e^{-c_\ell}\cdot \frac{e^{-V^2/\log\log T}}{\sqrt{\log\log T}}$ concluding the proof.
%
\end{proof}


\section{Applications to Short Intervals}
In \cite[Corollary 1.3]{ArgBai23}, it was shown that a large deviation estimate together with an appropriate discretization yields a probabilistic upper bound for the maximum of $|\zeta|$ in an interval 
of length $(\log T)^{\gamma}$ for some exponent $\gamma\geq 0$. The unconditional result was then limited to $\gamma<3$, because of the restriction $k<2$ on the values $V\sim k\log\log T$.
The same proof together with now Theorem \ref{thm: main} extends this bound conditionally to all $\gamma\geq 0$:
\label{sect: short}
\begin{cor}
	\label{cor: max}
	Assume the Riemann Hypothesis is true. Let $\gamma\geq0$ and $y>0$ such that $y=\oo\left(\frac{\log\log T}{\log\log\log T}\right)$. We have
	\begin{equation}
		\label{eqn: max}
		\max_{|h|\leq (\log T)^\gamma}|\zeta(1/2+\ii t+\ii h)|\leq e^y\frac{(\log T)^{\sqrt{1+\gamma}}}{(\log\log T)^{1/(4\sqrt{1+\gamma})}},
	\end{equation}
	for all $t\in[T,2T]$ except on a set of Lebesgue measure $\ll e^{-2y\sqrt{1+\gamma}} e^{-y^2/\log\log T}$.
\end{cor}
The bound is expected to be sharp for $\gamma>0$. The case $\gamma=0$ corresponds to the Fyodorov-Hiary-Keating Conjecture and was proved in \cite{ArgBouRad20}. 
The large deviation estimate can also be used exactly as in \cite[Corollary 1.4]{ArgBai23} to yield a conditional upper bound on the moments in short intervals.
\begin{cor}
	\label{cor: subcritical}
	Assume the Riemann Hypothesis is true. Let $\gamma\geq0$. For all $\beta\geq 0$, we have for $A>1$
	\begin{equation}
		\label{eqn: subcritical}
		\int_{|h|\leq (\log T)^\gamma} |\zeta(1/2+\ii t+\ii h)|^\beta\rd h\leq A (\log T)^{\frac{\beta^2}{4}+\gamma},
	\end{equation}
	for all $t\in [T,2T]$ except possibly on a subset of Lebesgue measure $\ll 1/A$. 
	
	For $\beta>\beta_c=2\sqrt{1+\gamma}$, a sharper bound holds:
	\begin{equation}
		\label{eqn: supercritical}
		\int_{|h|\leq (\log T)^\gamma} |\zeta(1/2+\ii t+\ii h)|^\beta\rd h\leq C_{A,\beta}\cdot  (\log\log T)^{-\tfrac{\beta}{2\beta_c}}\cdot (\log T)^{\tfrac{\beta_c}{2}\beta -1},
	\end{equation}
	for all $t\in [T,2T]$ except possibly on a subset of Lebesgue measure $\ll 1/A$, where $C_{A,\beta}$ is an explicit constant dependent on $A$ and $\beta$.
\end{cor}
The result should be sharp for $\gamma>0$, up to constant. In the critical case $\beta_c=2$ for $\gamma=0$, Harper in \cite{Har19} proved that the moment is stochastically bounded by $\log T/\sqrt{\log\log T}$. 
See also \cite{ArgHam24} for a different proof. 

\begin{appendix}
\section*{Appendix}
\section{Moments of $S_\ell^{(j)}$}
\label{sect: moments}
The moments of the partial sums \eqref{eqn: S} can be estimated by comparison with the Steinhaus model. 
Namely, consider $(X(p), p \text{ primes})$ independent random variables uniformly distributed on the unit circle. For $n\in \N$, define the random variable $X(n)$ multiplicatively according to the prime factorization of $n$. 
Then the mean-value theorem for Dirichlet polynomials, see for example  \cite{MonVau74}, can be stated as follows: for a sequence of complex numbers $(b_n, n\geq 1)$, and $\tau$ uniform on $[T, 2T]$ 
\begin{equation}
\label{eqn: MV}
\E\left[\Big|\sum_{n\leq N} b_n n^{-\ii \tau}\Big|^2\right]=\Big (1+\OO(N/T)\Big)\E\left[\Big|\sum_{n\leq N} b_n X(n)\Big|^2\right]=\Big (1+\OO(N/T)\Big)\sum_{n\leq N} |b_n|^2.
\end{equation}
In particular, this directly implies that for two Dirichlet polynomials with complex coefficients $(b_n, n\geq 1)$ and $(c_n, n\geq 1)$
\begin{equation}
\label{eqn: MV real}
\begin{aligned}
\E\left[\Big(\sum_{n\leq N} b_n n^{-\ii \tau}\Big)\overline{\Big(\sum_{n\leq N} c_n n^{-\ii \tau}\Big)}\right]&=\E\left[\Big(\sum_{n\leq N}b_nX(n)\Big)\overline{\Big(\sum_{n\leq N}c_nX(n)\Big)}\right]+\OO\Big(\frac{N}{T}\sum_{n\leq N} |b_n|^2+ |c_n|^2\Big).\\
&=\sum_{n\leq N}b_n\overline{c_n}+\OO\Big(\frac{N}{T}\sum_{n\leq N} |b_n|^2+ |c_n|^2\Big).
\end{aligned}
\end{equation}
These formulas can be applied to get the variance of the real sums $S_\ell^{(j)}$ defined in \eqref{eqn: S}: for $j\geq \ell$, 
\begin{equation}
\label{eqn: variance}
\begin{aligned}
\E[(S_\ell^{(j)})^2]&=\frac{1}{2}\sum_{p\leq T_\ell}\frac{a_j(p)^2}{p^{2\sigma_j}}+\frac{1}{4}\sum_{p^2\leq T_\ell} \frac{a_j(p^2)^2}{p^{4\sigma_j}} +\oo(1)\\
&\leq \frac{1}{2}\log\log T_\ell +1 ,
\end{aligned}
\end{equation}
by Mertens's estimates and the definition of $\sigma_j$ and $a_j$. The covariance at two abscissae $i,j\geq \ell$ is computed similarly
\begin{equation}
\label{eqn: covariance}
\begin{aligned}
\left|\E[S_\ell^{(i)}S_\ell^{(j)}]\right|&=\frac{1}{2}\left|\sum_{p\leq T_\ell}\frac{a_i(p)a_j(p)}{p^{\sigma_i+\sigma_j}}\right|+\OO(1)\\
&=\frac{1}{2}\sum_{p\leq T_\ell}\frac{1}{p^{\sigma_i+\sigma_j}}+\OO\Big(\sum_{p\leq T_\ell}\frac{\log p}{\log T_i}+\frac{\log p}{\log T_j}\Big)+\OO(1)=\frac{1}{2}\log\log T_\ell +\OO(1).
\end{aligned}
\end{equation}

The comparison with Steinhaus random variables can also be used to obtain upper bound for the moments of $S_\ell^{(j)}$. 
The following is Lemma 16 of~\cite{ArgBouRad20}. Note that in their case $a_j(p)=1$ and $\sigma_j=1/2$, but the same upper bound still holds.
\begin{lem}\label{lem: Gaussian moments real}
Let $S_\ell^{(j)}, S_k^{(j)}$ be as in \eqref{eqn: S} for $1\leq k<\ell\leq j$. For any $q\in \N$ with $T_\ell^{2q}\leq T^{1/4}$, we have
  \begin{equation}
  \label{eqn: moment bound}
  \mathbf{E}[\big(S_\ell^{(j)} - S_k^{(j)}\big)^{2q}] \ll \frac{(2q)!}{2^q q!} \left(\frac{\log \beta_\ell - \log \beta_k}{2}\right)^{q}.
  \end{equation}
  Moreover, there exists $C>0$ such that for $\ell=1$ and  $2q\leq \frac{\log T}{\log T_1}$,
  \begin{equation}\label{eqn: moment bound sqrt}
    \mathbf{E}[\big(S_1^{(j)}\big)^{2q}] \ll \sqrt{q}\frac{(2q)!}{2^q q!} \left(\frac{\log\log T_1}{2}+C\right)^{q}.
  \end{equation}
\end{lem}
The following upper bound for the moment generating function of the Steinhaus version of \eqref{eqn: S} is useful, and can be obtained by direct computation, see for example \cite[Lemma 15]{ArgBouRad20}:
for any fixed $\lambda\in \R$, and $0\leq k<\ell\leq \mathcal L$,
\begin{equation}
\label{eqn: MGF steinhaus}
\E\Big[\exp\Big( \lambda\sum_{T_k<p\leq T_\ell}\frac{a_j(p)}{p^{\sigma_j}}\re X(p)+\frac{\lambda}{2} \sum_{T_k<p^2\leq T_\ell}\frac{a_j(p^2)}{p^{2\sigma_j}}\re X(p^2) \Big)\Big]
\ll_\lambda \left(\frac{\log T_\ell}{\log T_k}\right)^{\frac{\lambda^2}{4}}.
\end{equation}

\section{Gaussian Comparison on the Good Events}
\label{sect: GC}
The following result appears as Lemma 2.6 in \cite{ArgBai23}. There, it was applied to a partition $(T_\ell, 1\leq \ell\leq \mathcal L)$ of the primes, which was sparser than the $(T^{\beta_\ell}, 1\leq \ell\leq \mathcal L)$ used herein.
We include a proof of the result for the partition \eqref{eqn: Tell} for completeness. The result holds with $\asymp$ instead of $\ll$, but the lower bound is not needed for our purpose.
We prove the lemma for the events $G_\ell$ defined in \eqref{eqn: G}. It applies equally well with the restrictions $L_\ell'$, $U_\ell'$ defined in \eqref{eqn: barriers prime}
as they only differ by a multiple of the $c_\ell$'s. 
\begin{lem}
\label{lem: 1 point}
  Let $1\leq \ell\leq \mathcal L$ and $\mathcal{Q}$ be of the form $\mathcal Q=(S_m^{(m)}-S_\ell^{(m)})^q$ for $q\in \N$, and $m\leq \mathcal L$ such that $T^{2q\beta_m}\leq T^{1/4}$.
  Consider the event $G_{\ell}(u)=G_{\ell-1}\cap\{S_\ell^{(j)}\in [u-1,u]\}$ with $u\in [L_\ell, U_{\ell}]$ and some $j\geq \ell$. We have
\begin{equation}
\label{eqn: appendix lem}
\E\Big[|\mathcal Q|^2 \1(G_{\ell}(u))\Big]\ll \E\Big[|\mathcal Q|^2\Big]\cdot \frac{e^{-u^2/\log\log T_\ell}}{\sqrt{\log\log T_\ell}}.
\end{equation}
\end{lem}

\begin{proof}
To simplify the notation in the proof, we write $[L_\ell, U_\ell]$ for the interval $[u-1,u]$. 
We also fix $\sigma_j$ in $S_\ell^{(j)}$ throughout and drop it from the notation. 
The first step is to express the restriction on the partial sums in terms of their increments
\begin{equation}
\label{eqn: increment}
Y_j=S_j-S_{j-1}, \quad 1\leq j\leq \mathcal \ell.
\end{equation}
We now partition $[L_j,U_j]$ in intervals of the form $[u_j, u_j+\Delta_j^{-1}]$ where
\begin{equation}
\label{eqn: Delta}
u_j\in \Delta_j^{-1}\Z \qquad \Delta_j=c_j,
\end{equation}
with $c_j$ given in \eqref{eqn: Bell}. By definition, we have $\sum_{j=1}^{\mathcal L} \Delta_j^{-1}=\OO(1)$.
We consider the values of $u_j$'s that are relevant for the restricted sums. Namely,  we consider $\ell$-tuples ${\bf u}=(u_1,\dots, u_\ell)$ in the set
\begin{equation}
\label{eqn: I}
\mathcal I=\{{\bf u} : \sum_{i=1}^j u_i \in [L_j-1,U_j],\ \forall j\leq\ell \}.
\end{equation}
We then have the inclusion
\begin{equation}
\label{eqn: inclusion}
      \{S_j\in [L_j, U_j],  j\leq \ell\}\subset \bigcup_{\mathbf{u} \in \mathcal I}\{Y_j \in [u_j, u_j+\Delta_j^{-1}], j\leq \ell\}.
\end{equation}
The definition of $\mathcal I$ implies the bound
\begin{equation}
\label{eqn: bound u}
|u_j|\leq 4\Delta_j. 
\end{equation}
With these considerations, the left-hand side of \eqref{eqn: appendix lem} is
$$
\leq \sum_{\bf u \in \mathcal I}\E\left[|\mathcal Q|^2 \prod_{j=1}^{\ell} \1\big(Y_j\in [u_j, u_j+\Delta_j^{-1}]\big)\right].
$$
We approximate the indicator function above by a polynomial as follows. 
Following Lemma 2 in \cite{ArgBai25}, for $a>2$, $X>1$ and $\Delta$ large enough, there exists a  polynomial $\mathcal D$ of degree smaller than $100 X\Delta^{3a}$ with $\ell$-th coefficient bounded by $\frac{(2\pi)^\ell}{\ell!} \Delta^{2a(\ell+2)}$ such that for $|x|\leq X$ we have
  \begin{equation}
    \label{eqn: D}
    \1(x\in [0,\Delta^{-1}])\Big(1-e^{-\Delta^{a-2}}\Big)\leq |\mathcal D(x)|^2\leq \1(x\in [-\Delta^{-a},\Delta^{-1}+\Delta^{-a}]) + e^{-\Delta^{a-2}}.
  \end{equation}
We pick $a=5$. From the restriction \eqref{eqn: bound u} for fixed $j\leq \ell$, we can pick $X=10 \Delta_j$. With these choices, the polynomial has degree less than $10^3 \Delta_j^{3a+1}\leq \Delta_j^{20}$.
We write $\mathcal D_j$ for this polynomial. 

We apply the mean-value theorem for Dirichlet polynomials to $\mathcal D_j(Y_j-u_j)^2$. 
Note that this is not a Dirichlet polynomial {\it per se} since $Y_j$ is the real part of a Dirichlet polynomial. 
However, by expanding the powers of $\mathcal D_j$, it can be re-expressed in a form where the theorem as in \eqref{eqn: MV real} can be applied. 
The length of the polynomial is then at most
\begin{equation}
\label{eqn: key condition}
(T_\ell)^{\Delta_\ell^{20}}=T^{\beta_\ell (\beta_\ell^{-20\gamma})},
\end{equation}
which is small enough for $\gamma<1/20$.
Define $\mathcal Y_j$ as in \eqref{eqn: increment} with the $p^{-\ii \tau}$'s replaced by independent Steinhaus random variables $X(p)$, i.e.~uniformly distributed on the unit circle.
This shows so far that
\begin{equation}
\label{eqn: MVapplied}
\E\left[|Q|^2 \prod_{j=1}^{\ell} \1\big(Y_j\in [u_j, u_j+\Delta_j^{-1}]\big)\right]\ll \E[|Q|^2]\cdot \prod_{j=1}^\ell\E\left[|\mathcal D_j(\mathcal Y_j-u_j)|^2\right].
\end{equation}

For each $j$, we decompose the expectation onto the events $|\mathcal Y_j-u_j|\leq\Delta_j^{6a}$ and $|\mathcal Y_j-u_j|>\Delta_j^{6a}$. 
On the first, we get from \eqref{eqn: D} that
$$
\E\left[|\mathcal D_j(\mathcal Y_j-u_j)|^2\1(|\mathcal Y_j-u_j|\leq\Delta_j^{6a})\right]=\PP(\mathcal Y_j-u_j\in [-\Delta_j^{-a},\Delta_j^{-1}+\Delta_j^{-a}])+e^{-\Delta_j^{a-2}}.
$$
By Gaussian approximation (see for example Lemma 8 and 9 in \cite{ArgBai25}), the probability is equal to $\PP(\mathcal N_j-u_j\in [-\Delta_j^{-a},\Delta_j^{-1}+\Delta_j^{-a}])+\OO(\exp(-ce^{t_{j-1}/2}))$ for $j\geq 2$, and with $\asymp$ for $j=1$. Here, $\mathcal N_j$ is defined as $Y_j$ with the $p^{-\ii \tau}$'s now replaced by independent complex standard Gaussian random variables $Z(p)$. In particular, since $|u_j|\Delta_j^{-a}\leq 1$, the overspill is negligible and we have
\begin{equation}
\label{eqn: overspill}
\PP( \mathcal N_j-u_j\in [-\Delta_j^{-a},\Delta_j^{-1}+\Delta_j^{-a}])=(1+\OO(\Delta_j^{1-a}))\cdot  \PP(\mathcal N_j-u_j\in [0,\Delta_j^{-1}]).
\end{equation}
Moreover, by a Gaussian estimate, this probability is larger than the error term $e^{-\Delta_j^{a-2}}$ for the choice $a=5$. Therefore, we have shown 
$$
\E\left[|\mathcal D_j(\mathcal Y_j-u_j)|^2\1(|\mathcal Y_j-u_j|\leq\Delta_j^{6a})\right]=(1+\OO(\Delta_j^{-1})) \cdot\PP(\mathcal N_j-u_j\in [0,\Delta_j^{-1}]).
$$
The contribution on the event  $|\mathcal Y_j-u_j|>\Delta_j^{6a}$ is estimated by the Cauchy-Schwarz inequality and a Chernoff bound
\begin{equation}
\label{eqn: D large}
\E\left[|\mathcal D_j(\mathcal Y_j-u_j)|^2\1(|\mathcal Y_j-u_j|>\Delta_j^{6a})\right]\ll \E\left[|\mathcal D_j(\mathcal Y_j-u_j)|^4\right]^{1/2}e^{-\frac{1}{4}\Delta_j^{6a}}
\end{equation}
using the estimate \eqref{eqn: MGF steinhaus} and the bound \eqref{eqn: bound u}. 
By adding all higher powers in $\mathcal D_j$, the fourth moment is bounded by
\begin{equation}
\label{eqn: D moment}
    \E\Big[ \Big ( \sum_{\ell \geq 1} \frac{(2\pi )^{\ell}}{\ell!}  2\Delta_j^{2a(\ell+2)}  (|\mathcal Y_j|+|u_j|)^{\ell} \Big )^4\Big ]  
    \leq   \, \E[\exp( 9\pi   \Delta_j^{2a}  (| \mathcal Y_j|+4\Delta_j))] \leq e^{\Delta_j^{5a}},
\end{equation}
which shows that \eqref{eqn: D large} is $\ll e^{-\frac{1}{8} \Delta_j^{6a}}$.
We conclude that
$$  
  \E\left[|\mathcal D_j(\mathcal Y_j-u_j)|^2\right]\leq (1+\Delta_j^{-1}) \cdot \PP(\mathcal N_j\in [u_j,u_j+\Delta_j^{-1}]).
 $$
 
It remains to take the product over $j$ of the above and apply the reverse inclusion \eqref{eqn: inclusion} to get that
\begin{equation}
\label{eqn: 1 point final}
\E\Big[|Q|^2 \1(G_{\ell}(u))\Big]\ll \E[|Q|^2]\cdot \PP\Big(\sum_{j\leq \ell} \mathcal N_j \in [u-2,u+1]\Big),
 \end{equation}
  where we dropped the intermediate restrictions on the  sums of $\mathcal N_j$. The Gaussian estimate holds using the variance estimate \eqref{eqn: variance}.
\end{proof}

\begin{lem}
\label{lem: 2 point}
Let $1\leq\ell\leq \mathcal L$. Let $u,v\in [L_\ell', U_\ell']$ as given in \eqref{eqn: barriers prime}. We have for $m,m'\geq \ell$
$$
\PP\Big(\{S_\ell^{(m)}\in [u-1,u], S_\ell^{(m')}\in [v-1,v]\}\cap G_{\ell-1}\Big)\ll \frac{e^{-\frac{(u+v)^2}{4\log\log T_\ell}}}{\sqrt{\log\log T}}  \cdot e^{-\frac{1}{8}(v-u)^2}.
$$
\end{lem}

\begin{proof}
The proof goes along the same lines as the one of  Lemma \ref{lem: 1 point}.
We write for the increments of $S_\ell^{(m)}$ and $S_\ell^{(m')}$ respectively as in \eqref{eqn: increment}.
Using the inclusion \eqref{eqn: inclusion}, the probability is bounded above by
$$
\sum_{{\bf u}\in \mathcal I(u),{\bf v}\in \mathcal I(v) }\PP\Big(Y_j\in [u_j,u_j+\Delta_j^{-1}, Y'_j\in [v_j,v_j+\Delta_j^{-1}]\ \forall j\leq \ell\Big)
$$
where
$\mathcal I(u)$ is defined as in \eqref{eqn: I} with $u_\ell\in [u-1,u]$. 
We now fix the two $\ell$-tuples ${\bf u,v}$. We use \eqref{eqn: D} to approximate the indicator functions for each $j$. This yields as in \eqref{eqn: MVapplied}
$$
\ll \E\left[\prod_{j\leq \ell}|\mathcal D_j(Y_j-u_j)|^2|\mathcal D_j(Y'_j-v_j)|^2\right]\ll \prod_{j\leq \ell}\E\left[|\mathcal D_j(\mathcal Y_j-u_j)|^2|\mathcal D_j(\mathcal{Y}'_j-v_j)|^2\right].
$$
The restriction $\{|Y_j-u_j|\leq \Delta_j^{6a}\}$ and $\{|Y'_j-v_j|\leq \Delta_j^{6a}\}$ can be inserted. 
Note that we have to apply the Cauchy-Schwartz inequality twice this time ending up with the 8th-moment of $\mathcal D_j$. 
This only changes the constant $9$ to $17$ in \eqref{eqn: D moment}. 
The expectation in the product then becomes
$$
\PP(\mathcal Y_j-u_j\in [-\Delta_j^{-a},\Delta_j^{-1}+\Delta_j^{-a}], \mathcal Y_j-v_j\in [-\Delta_j^{-a},\Delta_j^{-1}+\Delta_j^{-a}])+e^{-\Delta_j^{a-2}}.
$$
A two-point Gaussian approximation can be applied as in \cite[Lemma 13]{ArgBouRad23}. 
(Note that whilst the error one would find by this method is too large for $j=1$, the Gaussian density up to constants, i.e.~$\asymp$, follows by a change of measure and a standard Berry-Esseen estimate.)
Proceeding as in \eqref{eqn: overspill}, we get similarly to \eqref{eqn: 1 point final} that
$$
\PP(\mathfrak S_\ell\in [u-2,u+1], \mathfrak S'_\ell\in [v-2,v+1]),
$$
where $(\mathfrak S_\ell, \mathfrak S'_\ell)$ is a Gaussian vector of mean $0$ with variance \eqref{eqn: variance} and covariance 
$$
\E[\mathfrak S_\ell \mathfrak S_\ell']=\frac{1}{2}\sum_{p\leq T_\ell}\frac{a_m(p)a_{m'}(p)}{p^{\sigma_m+\sigma_{m'}}}+\frac{1}{4}\sum_{p^2\leq T_\ell}\frac{a_m(p^2)a_{m'}(p^2)}{p^{2\sigma_m+2\sigma_{m'}}}.
$$
Writing the probability density function of $(\mathfrak S_\ell, \mathfrak S'_\ell)$ explicitly and using the change of variables $\frac{1}{2}(\mathfrak S_\ell+ \mathfrak S'_\ell)$, $\frac{1}{2}(\mathfrak S_\ell-\mathfrak S'_\ell)$ 
then gives
$$
\ll \PP(\mathfrak S_\ell\in [u-1,u+2], \mathfrak S'_\ell\in [v-1,v+2])
\asymp \frac{e^{-\frac{(u+v)^2}{4t_\ell}}}{\sqrt{\log\log T}} \cdot e^{-\frac{1}{4C}(u-v)^2},
$$
where 
$$
C=\E[(\mathfrak S_\ell -\mathfrak S_\ell')^2]=\E[\mathfrak S_\ell^2]+\E[ {\mathfrak S_\ell'}^2]-2\E[\mathfrak S_\ell \mathfrak S_\ell'].
$$
We claim that
$
\Big|\E[\mathfrak S_\ell^2]-\E[\mathfrak S_\ell \mathfrak S_\ell']\Big|\leq 2,
$
which will conclude the proof. 
Indeed, by definition
$$
\Big|\E[\mathfrak S_\ell^2]-\E[\mathfrak S_\ell \mathfrak S_\ell']\Big|\leq \frac{1}{2}\sum_{p\leq T_\ell}\frac{a_m(p)}{p^{\sigma_m}}\cdot \Big|\frac{a_m(p)}{p^{\sigma_m}}-\frac{a_{m'}(p)}{p^{\sigma_{m'}}}\Big|+\frac{1}{6}.
$$
The absolute value is smaller than
$$
\Big|\frac{1}{p^{\sigma_m}}-\frac{1}{p^{\sigma_{m'}}}\Big|+\frac{1}{p^{\sigma_m}}\frac{\log p}{\log T_m}+\frac{1}{p^{\sigma_{m'}}}\frac{\log p}{\log T_{m'}}.
$$
Now, Mertens's first theorem yields
$$
\frac{1}{\log T_m}\sum_{p\leq T_\ell}\frac{\log p}{p}\leq 1+\frac{1}{100},
$$
since $T_m\geq T_\ell$. 
Assuming without loss of generality that $m>m'$, we get similarly
$$
\sum_{p\leq T_\ell} \frac{1}{p^{2\sigma_m}}\left(1-p^{\sigma_{m}-\sigma_{m'}}\right)\leq (\sigma_{m'}-\sigma_{m})(\log T_\ell+2)\leq \frac{1}{2}+\frac{1}{100}.
$$
Putting these estimates together yields the claim.
\end{proof}
\end{appendix}

\bibliographystyle{alpha}
\bibliography{ld_on_rh.bib}

\end{document}